\documentclass[10pt]{amsart}

\newcommand{\cal}[1]{\mathcal{#1}}
\usepackage{john}

\theoremstyle{theorem}
\newtheorem{notation}[theorem]{Notation}
\usepackage{hyperref}

\usepackage[displaymath, mathlines]{lineno}
\newcommand*\patchAmsMathEnvironmentForLineno[1]{%
  \expandafter\let\csname old#1\expandafter\endcsname\csname #1\endcsname
  \expandafter\let\csname oldend#1\expandafter\endcsname\csname end#1\endcsname
  \renewenvironment{#1}%
     {\linenomath\csname old#1\endcsname}%
     {\csname oldend#1\endcsname\endlinenomath}}%
\newcommand*\patchBothAmsMathEnvironmentsForLineno[1]{%
  \patchAmsMathEnvironmentForLineno{#1}%
  \patchAmsMathEnvironmentForLineno{#1*}}%
\AtBeginDocument{%
\patchBothAmsMathEnvironmentsForLineno{equation}%
\patchBothAmsMathEnvironmentsForLineno{align}%
\patchBothAmsMathEnvironmentsForLineno{flalign}%
\patchBothAmsMathEnvironmentsForLineno{alignat}%
\patchBothAmsMathEnvironmentsForLineno{gather}%
\patchBothAmsMathEnvironmentsForLineno{multline}%
}

\numberwithin{theorem}{section}

\newcommand{\CS}{\operatorname{CS}}
\newcommand{\CSk}{\operatorname{CS}^k}
\newcommand{\CSw}{\operatorname{CS}^w}

\newcommand{\GM}{\operatorname{GM}}
\newcommand{\ssfy}{\operatorname{ss}}
\newcommand{\bfloor}[1]{\left\lfloor#1\right\rfloor}

\makeatletter
\newcommand{\customlabel}[2]{%
   \protected@write \@auxout {}{\string \newlabel {#1}{{#2}{\thepage}{#2}{#1}{}} }%
   \hypertarget{#1}{}
}
\makeatother

\usepackage[dvipsnames]{xcolor}

\usepackage{float}
\restylefloat{table}

\usepackage[french,english]{babel}

\numberwithin{equation}{section}
\title[Upper bounds for constant slope $p$-adic families]{Upper bounds for constant slope $p$-adic families of modular forms}
\author{John Bergdall}
\date {\today}
\address{John Bergdall\\Department of Mathematics\\Bryn Mawr College\\ Bryn Mawr, PA 19010\\USA}
\email{jbergdall@brynmawr.edu}

\usepackage[all,cmtip]{xy}

\begin{document}
\begin{abstract}
We study $p$-adic families of eigenforms for which the $p$-th Hecke eigenvalue $a_p$ has constant $p$-adic valuation (``constant slope families''). We prove two separate upper bounds for the size of such families. The first is in terms of the logarithmic derivative of $a_p$ while the second depends only on the slope of the family. We also investigate the numerical relationship between our results and the former Gouv\^ea--Mazur conjecture.
\end{abstract}
%

\maketitle

The purpose of this article is to prove new results, the first of their kind, on the sizes of $p$-adic families of modular forms. Here, we describe our results in the case of tame level 1. See the text for full generality and precise definitions and theorems.

Suppose that $k\geq 2$ and $f = \sum a_n(f) q^n$ is a normalized Hecke eigenform of weight $k$ and level $\Gamma_0(p)$ where $p$ is a fixed prime. If $f$ is of non-critical slope, meaning the $p$-adic valuation $v_p(a_p(f))$ of $a_p(f)$ is less than $k-1$, then work of Hida (\cite{Hida-GaloisRepresentationsIntoLambda}) and Coleman (\cite{Coleman-pAdicBanachSpaces}) implies that there is a formal $q$-expansion
\begin{equation*}
\mathbf f = \sum a_n(\kappa)q^n
\end{equation*}
where the $a_n(\kappa)$ are rigid analytic functions in $\kappa$, at $\kappa = k$ the $q$-expansion $\mathbf f_k := \sum a_n(k)q^{n}$ is the same as $f$, and for $\kappa = k'$ a sufficiently large integer congruent to $k$ modulo $p-1$, $\mathbf f_{k'}$ is also the $q$-expansion of a normalized eigenform of weight $k'$ and level $\Gamma_0(p)$. The common domain of the functions $a_n(\kappa)$ is implicit in $\mathbf f$; it is an open affinoid subdomain $U$, containing $k$, inside the $p$-adic weight space. 

In general, we refer to $\mathbf f$ as a $p$-adic family ``passing through $f$.'' A family has constant slope if $\kappa \mapsto v_p(a_p(\kappa))$ is constant over the domain. This can always be arranged by shrinking the domain, so for a given $f$ there is a smallest integer $\CS(f) \geq 0$ such that a constant slope family passing through $f$ exists on a domain containing every integer $k' \congruent k \bmod (p-1)p^{\CS(f)}$. We call $\CS(f)$ the ``constant slope (valuation) radius'' of $f$. Bounding $\CS(f)$ from below is the same as bounding the size of $\mathbf f$ from above. Our first result is such a bound. We prove that
\begin{equation}\label{eqn:apprime-bound-intro}
v_p\left({a_p(f) \over a_p'(k)}\right) \leq \CS(f)
\end{equation}
where $a_p'(k) = {\mder a_p/\kappa}|_{\kappa=k}$ is the derivative of $a_p$ over a $p$-adic family containing $f$ (see Theorem \ref{theorem:L-invariant-in-text}). The proof of \eqref{eqn:apprime-bound-intro} is classical $p$-adic analysis after precise definitions are made, and the argument applies equally to  non-classical $p$-adic eigenforms.  

A special case of \eqref{eqn:apprime-bound-intro}  is realized by applying a famous theorem of Greenberg and Stevens (and its generalizations) to rewrite \eqref{eqn:apprime-bound-intro} as
\begin{equation}\label{eqn:Linv-bound-intro}
-v_p\left(\mathscr L_f\right) + v_p(2) \leq \CS(f)
\end{equation}
where $\mathscr L_f \in \bar \Q_p$ is the $\mathscr L$-invariant of $f$ if, say, $f$ is a newform. This allows one, for instance, to rule out the existence of very large $p$-adic families in many cases. On the other hand, the bound \eqref{eqn:apprime-bound-intro} also visually takes the form
\begin{equation}\label{eqn:gm-error-statement}
h - \text{``error''} \leq \CS(f)
\end{equation}
where $h = v_p(a_p(f))$ is the slope of $f$. This matches the numerical calculations that led to the Gouv\^ea--Mazur conjecture (\cite{GouveaMazur-FamiliesEigenforms}), so it is interesting to study how tight \eqref{eqn:apprime-bound-intro} is. It cannot be tight in general because, for instance, there exists newforms whose $\mathscr L$-invariants have positive valuation (in which case ``error'' is quite large!).

One feature of \eqref{eqn:apprime-bound-intro} is that it depends on $f$ and not just the slope $h$. This is notable because there is an implicit understanding among experts that the magnitude of $h$ is a direct obstruction to the existence of a family with constant slope $h$. Roughly, larger $h$'s should correspond to smaller families.  Motivated by this, we give a second bound  of the form
\begin{equation}\label{eqn:log-bound-intro}
\floor{m_p(h)} \leq \CS(f)
\end{equation}
where $m_p(h)$ is a non-negative function that grows like $\log h$ as $h \goto \infty$ (see Theorem \ref{thm:slope-dependent}). Actually, we will only produce a non-trivial bound if $p > 3$.

Unlike \eqref{eqn:apprime-bound-intro}, it is crucial that $f$ be in a family of eigenforms of level $\Gamma_0(p)$ for \eqref{eqn:log-bound-intro} to hold; there is no bound like \eqref{eqn:log-bound-intro} if we allow all $p$-adic eigenforms of slope $h$ at once because of ``spectral halos.''  To illustrate this, let us sketch an argument that is not complete. Gouv\^ea has conjectured  (\cite{Gouvea-WhereSlopesAre}) there is a $0\%$ chance that the slope of a weight $k$  eigenform on $\Gamma_0(p)$ lies strictly between ${k-1\over p+1}$ and ${k-2\over 2}$. Suppose that not only is Gouv\^ea's conjecture true but that it also can be strengthened to say that absolutely no slopes appear in that range. For fixed $h$ and $k$, among the $k'\geq 2$ satisfying $k' \congruent k \bmod p-1$ and ${k'-1\over p+1} < h < {k'-2\over 2}$ there is at least one with $v_p(k'-k)$ roughly $\log h$ (or higher), and our assumption implies that no constant slope family of slope $h$ passes over such $k'$. The argument fails since, of course, Gouv\^ea's conjecture is completely open, but also since the strengthening we considered is likely true only for primes Buzzard has called $\SL_2(\Z)$-regular (\cite{Buzzard-SlopeQuestions}).

The previous argument assumes certain slopes do not appear at all in certain weights. The key point in salvaging it is to use that a $p$-adic family of eigenforms comes equipped with a $p$-adic family of Galois representations and the ``mod $p$ reduction'' is constant over the family. The substitute for Gouv\^ea's global conjecture is a purely local version (Theorem \ref{thm:galois-rep-thm} in the text) that rules out lifts of the mod $p$ reduction of a fixed crystalline representation to crystalline representations, in other weights, with prescribed slope. This result, which we do not describe further but consider it of independent interest, is  an application of a theorem proven by Berger, Li, and Zhu (\cite{BergerLiZhu-SmallSlopes}). Their result is valid for any $p$, but unfortunately the way we apply it gives information only if $p > 3$.

In summary, if $f$ has slope $h$ then we have (for $p>3$) proven a bound
\begin{equation}\label{eqn:summary-bound}
\max\left(\floor{m_p(h)}, -v_p(a_p'(k)/a_p(f)\right) \leq \CS(f)
\end{equation}
for some explicit function $m_p(h)$. At the end of this article, we will give some numerical evidence that {\em this} bound is near to being tight, and we will discuss the relationship between \eqref{eqn:summary-bound} and the Gouv\^ea--Mazur conjecture.

We end by noting that automorphic and Galois-theoretic methods playing dual roles is typical for studies of deformation questions within the ``Langlands program.'' In the situation at hand, Wan (\cite{Wan-GouveaMazur}) used ($p$-adic) automorphic methods to produce upper bounds for quantities resembling $\CS(f)$ (thus lower bounds on the sizes of $p$-adic families) almost twenty years ago. As far as we know, those results have not been improved upon. By contrast, we prove bounds in the opposite direction and our bounds are either proven by Galois-theoretic considerations indicated above or may be interpreted as Galois-theoretic quantities. It is our hope that this line of inquiry opens new perspectives on the  problem originally raised by Gouv\^ea and Mazur. For further discussion about Galois representations and questions on slopes of modular forms, see \cite{BuzzardGee-Slopes}.

\subsection*{Organization}
This article comprises multiple short sections. Sections \ref{sec:defn-1st-theorem} through \ref{sec:log-derive-bound} concern the the first bound above and they also serve to set the notation and clarify the hypotheses. In Section \ref{sec:Linvariants} we discuss $\mathscr L$-invariants, whereas Section \ref{section:emerton} contains a non-trivial example where \eqref{eqn:apprime-bound-intro} is an equality. Sections  \ref{section:crys-lifts} through \ref{sec:optimization} are concerned with the second bound. In Section \ref{section:gouvea-mazur-comparison}, we examine the relationship between our results and the Gouv\^ea--Mazur conjecture.

\subsection*{Notations}
The letter $p$ always means a fixed prime number. We also need a second integer $N \geq 1$ that is assumed to be co-prime to $p$. 

We write $\bar \Q_p$ for an algebraic closure of $\Q_p$, $\C_p$ for its completion, and we normalize the $p$-adic valuation on $\C_p$ so that $v_p(p) = 1$. Throughout the paper we measure everything according to valuations, rather than norms. 

If $\bar \Q$ denotes the algebraic numbers in $\C$, then we also fix an embedding $\bar \Q \subset \bar \Q_p$, allowing us to compute $v_p(\alpha)$ where $\alpha$ is an algebraic integer, construct $\bar \Q_p$-linear Galois representations associated to eigenforms, etc.\

Throughout the paper, rigid analytic spaces are taken in the language of Tate's rigid analytic geometry, as opposed to the theories developed by Berkovich or Huber.

\subsection*{Acknowledgements}
The research reported on here was partially supported by NSF award DMS-1402005. The author thanks James Newton, Robert Pollack, Sandra Rozensztajn, and an anonymous referee for helpful discussions and comments. There are some computer calculations given below that were made by Pollack; we also thank him for access to this data.  In addition, during the elaboration of this work we also benefited from short visits to, and the hospitality at, Imperial College (London), the Institut des Hautes \'Etudes Scientifiques (Bures-sur-Yvette), and the Max-Planck-Institut f\"ur Mathematik (Bonn). The staff and members of these institutions are duly thanked.

\section{A consequence of the maximum modulus principle}\label{sec:defn-1st-theorem}

Suppose that $m \geq 0$ is a rational number. Consider the one-dimensional Tate algebra over $\C_p$ given by
\begin{equation*}
\C_p\langle wp^{-m} \rangle = \set{\sum_{i\geq 0} r_i w^i \st r_i \in \C_p \text{ and } v_p(r_i) + mi \goto +\infty  \text{ as $i\goto +\infty$}}.
\end{equation*}
This is a $\C_p$-Banach algebra with the Gauss valuation
\begin{equation}\label{eqn:gauss-norm}
v^{(m)}(F) = \inf\set{v_p(r_i) + mi \st i \geq 0}.
\end{equation}
For $w_0 \in \C_p$ we write $\B(w_0,m)$ for the affinoid space defined by $v_p(w-w_0) \geq m$. Then, there is a canonical isomorphism between $\C_p\langle wp^{-m}\rangle$ and the ring $\mathcal O(\B(w_0,m))$ of rigid analytic functions on $\B(w_0,m)$. We write $\B^{\circ}$ for the open $p$-adic unit disc
\begin{equation*}
\B^{\circ} = \set{w \st v_p(w) > 0} = \bigunion_{m > 0} \B(0,m).
\end{equation*}
\begin{definition}\label{defn:radius-ball}
Suppose that $W \subset \B^{\circ}$ is an affinoid open subdomain and $w_0 \in W$. The radius of $W$ at $w_0$ is 
\begin{equation*}
m_{w_0}(W) = \inf\set{m > 0 \st \B(w_0,m) \subset W}.
\end{equation*}
\end{definition}
Our only result here is an application of the maximum modulus principle.
\begin{proposition}\label{prop:constant-slope-approximation}
Let $W \subset \B^{\circ}$ be an affinoid open subdomain and $F \in \mathcal O(W)$ be non-zero and such that $w \mapsto v_p(F(w))$ is constant on $W$. Then, for any $w_0 \in W$,  we have
\begin{equation*}
v_p(F(w_0)) - v_p(F'(w_0)) \leq m_{w_0}(W).
\end{equation*}
\end{proposition}
\begin{proof}
It is sufficient to show the result for $W = \B(0,m)$ and $w_0 = 0$ (so $m_{w_0}(W) = m$). Consider any $F = \sum r_iw^i  \in  \C_p\langle w p^{-m} \rangle$. The maximum modulus principle (\cite[Proposition 5.1.4/3]{BGR}) implies that
\begin{equation}\label{eqn:sup-norm}
\inf_{w \in \B(0,m)} v_p(F(w)) = v^{(m)}(F).
\end{equation}
Since $v_p(F(w))$ is constant on $\B(0,m)$, the left-hand side of \eqref{eqn:sup-norm} is equal to $v_p(F(0))$. On the other hand, \eqref{eqn:gauss-norm} gives
\begin{equation*}
v^{(m)}(F) \leq v_p(r_1) + m = v_p(F'(0)) + m.
\end{equation*}
Rearranging the inequalities, the result follows.
\end{proof}

\begin{remark}\label{rmk:constant-slope-root-test}
In the situation of the proposition, let $M$ be the infimum over all $m$ such that $F$ has constant slope on $\mathbf B(w_0,m)$. Then, in fact, 
$$
M = \inf_{i\geq 1}{1\over i}\left(v_p(F(w_0))-v_p(r_i)\right),
$$
which amusingly appears to be a ``root test'' for the (valuation) radius of the largest disc around $w_0$ on which $v_p(F(w))$ is constant. The inequality $M \geq (\dotsb)$ follows from the same proof as Proposition \ref{prop:constant-slope-approximation}. We omit the (straightforward) proof of the reverse inequality.
\end{remark}

\section{$p$-adic eigenforms}\label{sec:padic-eigenforms}
Write $\mathcal W$ for the $p$-adic weight space over $\C_p$. Thus a $\C_p$-point of $\mathcal W$ is a continuous character $\kappa: \Z_p^\times \goto \C_p^\times$. Denote by $\Delta \subset \Z_p^\times$ the multiplicative torsion subgroup and note that $\Delta \times 1+2p\Z_p \simeq \Z_p^\times$ and $1+2p\Z_p$ is pro-cyclic. Given a character $\epsilon$ of $\Delta$, we say that $\kappa \in \mathcal W$ is of type $\epsilon$ if $\kappa|_{\Delta} = \epsilon$. In this way, we can write $\mathcal W$ as a union 
\begin{equation*}
\mathcal W = \bigunion_{\epsilon} \mathcal W_{\epsilon}
\end{equation*}
of connected components $\mathcal W_{\epsilon}$ consisting exactly of those $\kappa$ of type $\epsilon$. If the type of $\kappa$ is fixed, then $\kappa$ is completely determined by its value $\kappa(\gamma)$ for some/any choice of generator $\gamma$ for $1+2p\Z_p$. We fix such a choice and then for $\kappa \in \mathcal W$ we define $w_\kappa = \kappa(\gamma)-1$. The association $\kappa \mapsto w_\kappa$ provides a coordinate chart $\mathcal W_{\epsilon} \simeq \B^{\circ}$, for any $\epsilon$, that depends on $\gamma$ only up to isometry. In particular, for each $\kappa \in \mathcal W$ and rational number $m > 0$ we have a well-defined open affinoid subdomain $\B(\kappa,m) := \B(w_\kappa,m) \subset \mathcal W$.

\begin{remark}\label{rmk:variable-change}
If $k \in \Z$, then it defines an element $k \in \mathcal W$ given by the character $z \mapsto z^k$. A direct computation shows that $v_p(w_k-w_{k'}) = 1 + v_p(2) + v_p(k-k')$.
\end{remark}

Now we turn towards modular forms. Assume for the rest of this article that $N\geq 1$ is an integer and $p \ndvd N$. We write $\Gamma_1(N)$, $\Gamma_0(p)$, and $\Gamma=\Gamma_1(N)\cap \Gamma_0(p)$ for the standard congruence subgroups of $\SL_2(\Z)$. For $k\geq 2$ we denote by $M_k(\Gamma_1(N))$ and $M_k(\Gamma)$ the space of weight $k$ modular forms of levels $\Gamma_1(N)$ and $\Gamma$. We will also use the notations $S_k(\Gamma_1(N))$ and $S_k(\Gamma)$ for the corresponding spaces of cuspforms.

We will use bold $\mathbf T$'s to stand for certain Hecke algebras. Specifically, $\mathbf T_N$ is the commutative $\mathbf Z$-algebra generated by symbols $T_\ell$, for $\ell \ndvd N$ a prime, and $\langle d \rangle$, for $d \in (\Z/N\Z)^\times$. Then, $\mathbf T_N$ acts by endomorphisms on $M_k(\Gamma_1(N))$ through the standard Hecke operators with the same notation. We write $\mathbf T_\Gamma$ for the same $\Z$-algebra except the symbol $T_p$ is replaced by $U_p$ and then it acts by endomorphisms on $M_k(\Gamma)$. As a convention, we shorten the phrase ``$f \in M_k(\Gamma)$ is a normalized eigenform for $\mathbf T_{\Gamma}$'' to ``$f \in M_k(\Gamma)$ is an eigenform'', or some variation on that (and similarly for $M_k(\Gamma_1(N))$).

If $g$ is an eigenform we write $a_\ell(g)$ for its $\ell$-th Hecke eigenvalue and $\psi_g$ for its nebentype character. When $g$ is further an element of $M_k(\Gamma_1(N))$, its $p$-th Hecke polynomial is defined to be $X^2 - a_p(g) X + \psi_g(p)p^{k-1}$. We may factor this polynomial as $(X-\alpha)(X-\beta)$, and the algebraic integers $\alpha$ and $\beta$ are called ($p$-)refinements of $g$. Given a refinement, say $\alpha$, we define
\begin{equation*}
g_\alpha(z) = g(z) - \beta g(pz).
\end{equation*}
The modular form $f = g_\alpha$ (which is called a ($p$-)stabilization) is then an eigenform in $M_k(\Gamma)$ whose Hecke eigenvalues under $T_\ell$ ($\ell \neq p$) and $\langle d \rangle$ are the same as $g$'s, but whose $U_p$-eigenvalue is $\alpha$. Because $N$ is prime to $p$, every eigenform in $M_k(\Gamma)$ is either a stabilization or new at level $p$ (a ``$p$-new'' eigenform).

\begin{definition}\label{defn:refined-eigenform-defns}
Let $f \in M_k(\Gamma)$ be an eigenform.
\begin{enumerate}
\item $f$ is called regular if either $f$ is $p$-new or $f = g_\alpha$ for some eigenform $g\in M_k(\Gamma_1(N))$ whose $p$-th Hecke polynomial has distinct roots.
\item The slope of $f$ is $v_p(a_p(f))$. The slope is non-critical (or $f$ has non-critical slope) if $v_p(a_p(f)) < k-1$.
\end{enumerate}
\end{definition}

\begin{remark}\label{rmk:regularity}
The regularity condition always holds if $k=2$ and follows in general from a conjecture of Tate (see \cite{ColemanEdixhoven-Semisimplicity}). It is also vacuous if $N = 1$ (\cite[Theorem 1]{Gouvea-WhereSlopesAre}).
\end{remark}

\begin{remark}\label{rmk:old-slope}
We will also sometimes refer to the slope $v_p(a_p(g))$ of an eigenform $g \in M_k(\Gamma_1(N))$. If $v_p(a_p(g)) < {k-1\over 2}$ then the two stabilizations $g_\alpha$ and $g_\beta$ are distinct and give regular eigenforms at level $\Gamma$; their slopes are given (in some order) by $v_p(a_p(g))$ and $k-1 - v_p(a_p(g))$.
\end{remark}

For each $p$-adic weight $\kappa$ we now write $M_{\kappa}^{\dagger}(\Gamma)$ for Coleman's space of overconvergent $p$-adic modular forms of tame level $\Gamma_1(N)$ and weight $\kappa$ (\cite{Coleman-pAdicBanachSpaces}). This space still has an action of $\mathbf T_{\Gamma}$ by endomorphisms and the subspace $S_{\kappa}^{\dagger}(\Gamma)$ of overconvergent $p$-adic cuspforms is Hecke stable. 

The eigencurve $\mathcal C_N$ of tame level $N$ (see \cite{ColemanMazur-Eigencurve,Buzzard-Eigenvarieties}) is the one-dimensional rigid analytic space over $\Q_p$ that parameterizes the $\T_{\Gamma}$-eigensystems appearing in the spaces $M_{\kappa}^{\dagger}(\Gamma)$ and that are non-vanishing at $U_p$ (finite slope eigensystems). Any eigenform $f \in M_k(\Gamma)$ naturally defines a finite slope eigenform in $M_{k}^{\dagger}(\Gamma)$ and thus a canonical point of $\mathcal C_N$. Given a finite slope eigenform $f \in M_{\kappa}^{\dagger}(\Gamma)$ we write $x_f$ for the corresponding point on $\mathcal C_N$. The map that sends $x_f \mapsto \kappa$ defines a canonical map $\kappa: \mathcal C_N \goto \mathcal W$ called the weight map. We also have natural morphisms
\begin{align*}
a_\ell &: \mathcal C_{N} \rightarrow \mathbf A^1 & \text{($\ell \ndvd N$ prime)}\\
\langle d \rangle &: \mathcal C_{N} \rightarrow \mathbf A^1 & \text{($\gcd(d,N) = 1$)}
\end{align*}
that record the Hecke eigensystem at a given point ($\mathbf A^1$ is the affine line). The function $a_p$ is non-vanishing on $\mathcal C_N$.

\begin{definition}
A classical point on $\mathcal C_N$ is a point of the form $x = x_f$ for some eigenform $f \in M_k(\Gamma)$.
\end{definition}
Classical points are ubiquitous on $\mathcal C_N$. For instance, points of integer weight have neighborhood bases in which the classical points are Zariski-dense.

We end this section by defining a certain hypothesis under which our results are most naturally stated. It is verified for many classical points on the eigencurve.
\begin{hypothesis*}[\'et]\customlabel{hyp:et}{\'et}
If $f \in M_{\kappa}^{\dagger}(\Gamma)$ is a finite slope eigenform, then it satisfies \eqref{hyp:et} if the weight map $\kappa: \mathcal C_{N} \goto \mathcal W$ is \'etale at $x_f$.
\end{hypothesis*}
We will often write ``$f \in M_\kappa^{\dagger}(\Gamma)$ satisfies \eqref{hyp:et},'' implicitly including the qualification that $f$ is finite slope eigenform.

\begin{proposition}\label{prop:p-refined-etale}
If $f\in M_k(\Gamma)$ is an regular eigenform of non-critical slope, then the weight map $\kappa$ is \'etale at $x_f$. Thus $f$ satisfies \eqref{hyp:et}.
\end{proposition}
\begin{proof}
This is essentially given by the two lemmas in \cite[Section 2.1.4]{Bellaiche-CriticalpadicLfunctions}, except for the restriction on the level at which $f$ is new. This takes a small amount of work to remove, so we explain the argument in full.

First, if $N' \dvd N$ let us write $\mathcal C_{N'}^{N}$ for the eigencurve parameterizing finite slope $\mathbf T_\Gamma$-eigensystems appearing in spaces $M_\kappa^{\dagger}(\Gamma')$ where $\Gamma' = \Gamma_1(N') \cap \Gamma_0(p)$. As observed at the start of \cite[Lemma 2.7]{Bellaiche-CriticalpadicLfunctions}, there is a canonical closed immersion $\mathcal C_{N'}^N \inject \mathcal C_N$. 

Let $f$ be as in our proposition. Then, we can find a unique $N' \dvd N$ and an eigenform $f' \in M_k(\Gamma')$, which is either new or the stabilization of a newform of level $\Gamma_1(N')$. The corresponding classical point $x_{f'} \in \mathcal C_{N'}^N$ maps to $x_f$ under $\mathcal C_{N'}^N \inject \mathcal C_N$. Since $f$ is regular of non-critical slope, so is $f'$. In addition, the generalized $\mathbf T_{\Gamma}$-eigenspace associated to $f'$ in $M_k(\Gamma')$ is one-dimensional (strong multiplicity one allows us to ignore the eigenvalues at primes dividing $N/N'$) and so the argument in \cite[Lemma 2.8]{Bellaiche-CriticalpadicLfunctions} immediately extends to show $\mathcal C_{N'}^N \goto \mathcal W$ is \'etale at $x_{f'}$. 

The previous paragraph reduces us to showing that $\mathcal C_{N'}^N \goto \mathcal C_N$ is \'etale at $x_{f'}$. Since it is a closed immersion, we need to show that in a sufficiently small neighborhood of $x_f$ in $\mathcal C_N$, every classical point arises from a Hecke eigensystem that is new of tame level $N'$. If $f$ is Eisenstein, it is ordinary since $f$ is assumed to have non-critical slope. This case can be handled explicitly, so we assume that $f$ is cuspidal. In that case the two-dimensional Galois representation associated with $f$ is absolutely irreducible and extends to a two-dimensional family of Galois representations is a neighborhood $U$ of $x_f$ on $\mathcal C_N$. If $Z \subset U$ is an irreducible component then the tame conductor of the Galois representations are constant at {\em classical} points in $Z$ (\cite{Saha-Conductors}). It follows that every classical point on $Z$ arises from level $N'$, finishing the proof. (Section \ref{sec:slope-bounds} contains further discussion of Galois representations.)
\end{proof}

\section{Constant slope radii and log derivatives}\label{sec:log-derive-bound}

\begin{definition}
Assume that $f \in M_{\kappa}^{\dagger}(\Gamma)$ satisfies \eqref{hyp:et}. A $p$-adic family passing through $f$ is an irreducible component $U$ of an affinoid neighborhood of $x_f$ in $\mathcal C_N$ such that $\kappa: U \goto \kappa(U)$ is a rigid analytic isomorphism onto an affinoid open subdomain $\kappa(U) \subset \mathcal W$.

We say a $p$-adic family $U$ has constant slope if $u \mapsto v_p(a_p(u))$ is constant.
\end{definition}

\begin{remark}
In this article, a $p$-adic family means explicitly that the weight function is a local coordinate of the family. It would be interesting to formulate an analog of Theorem \ref{theorem:L-invariant-in-text} that takes into account the interesting possibility of ramification of the weight map.
\end{remark}

There is an obvious way to construct a new $p$-adic family from an old one, by restricting the domain. In particular, since $a_p$ is non-vanishing on the eigencurve, any $f \in M_{\kappa}^{\dagger}(\Gamma)$ satisfying \eqref{hyp:et} must have some constant slope $p$-adic family $U$ passing through it (the construction of the eigencurve implies that $\kappa(U)\subset \mathcal W$ is necessarily an affinoid open subdomain for $U$ small enough). So, the following definition is well-posed.
\begin{definition}\label{defn:CS-defn}
Assume that $f \in M_{\kappa}^{\dagger}(\Gamma)$ satisfies \eqref{hyp:et}. Then,
\begin{equation*}
\CSw(f): = \inf \set{m_{\kappa}(\kappa(U)) \st U \text{ is a constant slope $p$-adic family through $f$}}.
\end{equation*}
\end{definition}
The notation $\CS$ indicates the phrase ``constant slope.'' We refer to $\CSw(-)$ as a/the constant slope radius. Let us stress again that our measurement is in terms of a valuation, rather than a norm. The superscript ``$w$'' indicates that we are measuring the radii according to the coordinate $w_\kappa$ as opposed to measuring directly in the weight variable $k$.

If $f \in M_{\kappa}^{\dagger}(\Gamma)$ is a finite slope eigenform satisfying \eqref{hyp:et} and $U$ is a $p$-adic family passing through $f$, with $W = \kappa(U)$, then the function $a_p \in \mathcal O(U)$ naturally defines an element $a_p \in \mathcal O(W)$. We can thus expand $a_p$ on $W$ as a power series in $w - w_\kappa$ 
\begin{equation*}
a_p(w) = a_p(w_\kappa) + a_p'(w_\kappa)(w-w_{\kappa}) + \dotsb.
\end{equation*}
Of course, $a_p(w_\kappa) = a_p(f)$.
\begin{theorem}\label{theorem:L-invariant-in-text}
If $f \in M_{\kappa}^{\dagger}(\Gamma)$ satisfies \eqref{hyp:et}, then
\begin{equation}\label{eqn:log-bound}
v_p(a_p(f)) - v_p(a_p'(w_\kappa)) \leq \CSw(f).
\end{equation}
\end{theorem}
\begin{proof}
If $U$ is a constant slope $p$-adic family passing through $f$ and $W = \kappa(U)$, then
\begin{equation*}
v_p(a_p(f)) - v_p(a_p'(w_\kappa)) \leq m_{\kappa}(W)
\end{equation*}
by Proposition \ref{prop:constant-slope-approximation}. Taking the infimum over such $U$ proves the theorem.
\end{proof}

Let us make some minor consistency checks. First, it is  possible that $a_p'(w_\kappa) = 0$ (as opposed to $a_p(f)$, which is always non-zero). In that case, the result is trivial because $0 \leq \CSw(f)$ by definition. More generally, Theorem \ref{theorem:L-invariant-in-text} is also trivial if the left-hand side of \eqref{eqn:log-bound} happens to be non-positive (cf.\ Example \ref{example:single-L-invariant-example}).

A second check we might use is Hida theory (where $v_p(a_p(f)) = 0$). Hida theory provides many examples where $\CSw(f) = 0$, i.e.\ examples where $f$ lives in a $p$-adic family that maps isomorphically onto an entire component of $p$-adic weight space. If that is so, our bound implies that $v_p(a_p'(w_\kappa)) \geq 0$. That is consistent because, in such examples, the function $a_p$ is defined by a power series with integral coefficients.

Finally, in the next two sections we will examine Theorem \ref{theorem:L-invariant-in-text} through estimates on $a_p'/a_p$ that are accessible for non-trivial reasons. It would be interesting to find an $f$ (if it exists) where taking into account higher derivatives, as in Remark \ref{rmk:constant-slope-root-test}, provided an improvement on Theorem \ref{theorem:L-invariant-in-text}. Likewise, in Theorem \ref{theorem:L-invariant-in-text} you could also replace $p$ by a prime $\ell \neq p$ provided $a_\ell(f)$ is non-zero. Could that improve the estimate for $\CS^w(f)$?

\section{Relationship with $\mathscr L$-invariants}\label{sec:Linvariants}
If $f$ is a cuspidal and $p$-new, then $f$ is regular and in fact its slope is ${k-2\over 2} < k-1$. So, Theorem \ref{theorem:L-invariant-in-text} applies to $f$. On the other hand, such an $f$ also has an $\mathscr L$-invariant $\mathscr L_f \in \bar \Q_p$ defined in \cite{Mazur-Monodromy} and known as the ``Fontaine--Mazur $\mathscr L$-invariant.''

\begin{lemma}\label{lemma:L-invariant}
If $f \in S_k(\Gamma)$ is a $p$-new eigenform, then
\begin{equation}\label{eqn:L-invariant}
v_p(\mathscr L_f) = 2v_p(2) + 1 + v_p\left({a_p'(w_\kappa)\over a_p(f)}\right).
\end{equation}
\end{lemma}
\begin{proof}
A generalization of a famous theorem of Greenberg and Stevens tells us
\begin{equation}\label{eqn:greenberg-stevens}
\mathscr L_f = - 2 a_p(f)^{-1} {\mder /k a_p(k)},
\end{equation}
where $a_p(-)$ is expanded as a series in the variable $k \in \Z_p$ (see \cite[Corollaire 0.7]{Colmez-L_invariants}). By Remark \ref{rmk:variable-change},
\begin{equation}\label{eqn:valuations-change}
 v_p\left(\mder /k a_p(k)\right) = v_p\left(a_p'(w_\kappa)\right) + 1 + v_p(2).
\end{equation}
Thus \eqref{eqn:L-invariant} follows from \eqref{eqn:greenberg-stevens} and \eqref{eqn:valuations-change}.
\end{proof}

By Remark  \ref{rmk:variable-change}, it seems reasonable to make the the following alternate normalization of constant slope radii.
\begin{definition}\label{defn:CSk}
$\CS^k(-) := \max(0,\CS^w(-) -1 - v_p(2))$.
\end{definition}
For instance, if $f \in M_k(\Gamma)$ is an eigenform, then $\CS^k(f) = m_0 \in \Z$ means that $f$ lives in a $p$-adic family over all the integers $k' \congruent k \bmod (p-1)p^{m_0}$. This is what we used in the introduction. The effect of the maximum is to restrict to measuring only portions of families lying over the most the smallest affinoid disc containing $\Z_p$, i.e.\ the most central part of the $p$-adic weight space.

\begin{theorem}\label{theorem:L-invariant-normalization}
If $f \in S_k(\Gamma)$ is a $p$-new eigenform, then 
\begin{equation*}
-v_p(\mathscr L_f) + v_p(2) \leq \CS^k(f).
\end{equation*}
\end{theorem}
\begin{proof}
This is immediate from Theorem \ref{theorem:L-invariant-in-text} and Lemma \ref{lemma:L-invariant}.
\end{proof}

\begin{example}\label{example:single-L-invariant-example}
Let $p = 5$, $N=1$, and $k=10$. In $S_{10}(\Gamma_0(5))$ there are three newforms $f_0,f_1,f_2$. Robert Pollack has computer programs that compute $\mathscr L$-invariants.\footnote{We do not explain Pollack's method here. It is  generally based on combining the Mazur--Tate--Teitelbaum conjecture with calculations of $p$-adic $L$-functions. The reader may also be interested in the recent works \cite{Graef-Linvariant, ABGT-Linvariants}.} Up to labeling, they satisfy $v_5(\mathscr L_{f_0}) = 2$ and $v_5(\mathscr L_{f_1}) = v_5(\mathscr L_{f_2}) = -2$. So, Theorem \ref{theorem:L-invariant-normalization} gives no information for $f_0$, but for $f = f_1$ or $f = f_2$ we learn that a constant slope family passing through $f$ must be restricted at least to weights $k' \congruent 10 \bmod 4 \cdot 5^2$. See Section \ref{section:gouvea-mazur-comparison} for further discussions of this data.
\end{example}

\begin{example}\label{example:breuil-mezard}
Suppose that $f \in S_k(\Gamma_0(Np))$ is new at $p$ and $v_p(\mathscr L_f) < 1 - {k-2\over 2}$. Then Theorem \ref{theorem:L-invariant-normalization} implies that ${k-2\over 2} - 1 < \CSk(f)$. The slope of $f$ is ${k-2\over 2}$ so, in spirit, $f$ is nearly a counter-example to the Gouv\^ea--Mazur conjecture. 

On the other hand, at least if $p > 2 $ and $2 < k < p+1$, Breuil and M\'ezard separately calculated (\cite[Corollary 4.3.3.1]{BreuilMezard-Multiplicities}) that $v_p(\mathscr L_f) < 1 - {k-2\over 2}$ is sufficient to guarantee that the mod $p$ Galois representation associated with $f$ is irreducible, even after restricting to a decomposition group at $p$. This is consistent with the yoga used by Buzzard and Calegari to find a counter-example to the Gouv\^ea--Mazur conjecture. 

We'll add that, according to Pollack's data, when $p=59$ and $N=1$ there is a newform in weight $k=16$ whose $\mathscr L$-invariant has valuation $-7 < 1 - {16-2\over 2}$.
\end{example}

\begin{remark}\label{remark:benois}
Benois (\cite{Benois-GreenbergL}) has given a definition, generalizing work of Greenberg (\cite{Greenberg-TrivialZeroes}), of an ``$\mathscr L$-invariant'' associated with any eigenform $f \in S_k(\Gamma)$ under a suitable Selmer group hypothesis (which is conjecturally always true). Mok showed in \cite{Mok-LInvariant} that \eqref{eqn:greenberg-stevens} still holds. So, we could have stated Theorem \ref{theorem:L-invariant-normalization} in this generality as well. We do not know, however, of any way to directly calculate the $\mathscr L$-invariants beyond the $p$-new cases. (We do know how to compute the logarithmic derivatives sometimes. See the proof of Proposition \ref{prop:2adic-example}.)
\end{remark}

\section{An example (after Emerton)}\label{section:emerton}

In Example \ref{example:single-L-invariant-example}, we saw that Theorem  \ref{theorem:L-invariant-in-text} sometimes yields no information. Our goal here is to give one non-trivial example where the bound in Theorem \ref{theorem:L-invariant-in-text} is provably an equality.

For this section, we suppose that $p=2$ and $N=1$. The vector space $S_{14}(\Gamma_0(2))$ is two-dimensional and spanned by newforms $f^{\pm}$ labeled according to whether $a_2(f^{\pm}) = \pm 2^6$ (each happens once). What is special about this situation is that Emerton has given (\cite{Emerton-Thesis}) explicit equations defining the lowest slope cuspidal families of the $2$-adic tame level one eigencurve  $\mathcal C$. Specifically, if $W = \B(w_{14},6)$, then Emerton constructed an irreducible region $U \subset \mathcal C$ on which $a_2$ has constant slope $6$, the points of $U$ correspond to the lowest slope cuspidal eigenforms at weights $\kappa \in W$, and $\kappa : U \goto W$ is a two-to-one cover ramified at two points. Especially, each of $f^{\pm}$ defines a point $x_{\pm}:= x_{f^{\pm}}$ on $U$ and Proposition \ref{prop:p-refined-etale} tells us $\kappa$ is \'etale at each $x_{\pm}$. So, $a_2$ has a series expansion around $w = w_{14}$ (the expansion depends on $\pm$ of course).

\begin{proposition}\label{proposition:emertons-result}
The ramification of ${\kappa}|_{U}$ occurs at two weights $\kappa$ with 
\begin{equation*}
v_2(w_\kappa - w_{14}) = 7.
\end{equation*}
In particular, $\CSw(f^{\pm}) = 7$.
\end{proposition}
\begin{proof}
This is contained in the proof of \cite[Lemma 4.13]{Emerton-Thesis}. There, an equation for $U$ is given and our claim follows from the estimate of ``$d_1$'' in the middle of the proof.
\end{proof}
We now present the following calculation, complementary to Proposition  \ref{proposition:emertons-result}, that shows that the bound in Theorem \ref{theorem:L-invariant-in-text} is tight.
\begin{proposition}\label{prop:2adic-example}
$v_2\left(a_2'(w_{14})\over a_2(f^{\pm})\right) = -7$.
\end{proposition}
\begin{proof}
One way to prove this is to compute $\mathscr L_{f^{\pm}}$ and then apply Lemma \ref{lemma:L-invariant}. Pollack did this for us and told us our proposition was correct. For the sake of completeness, let us give an alternative calculation.

Write $P(w,t) = 1 + \sum c_i(w)t^i \in \Z_2[\![w,t]\!]$ for the characteristic series of the $U_2$-operator acting on overconvergent $2$-adic cuspforms of even weight (those $\kappa$ such that $\kappa(-1) = 1$). For such weights $\kappa$, $P(w_\kappa,t) \in \C_2[\![t]\!]$ is equal to $\det(1 - tU_2|S_\kappa^{\dagger}(\Gamma))$. As discussed in \cite[Section IV]{ColemanStevensTeitelbaum-Experiments}, we have
\begin{equation}\label{eqn:log-derivative}
{a_2'(w_{14})\over a_2(f^{\pm})} = a_2(f^{\pm}){{\mder /w}P(w,t) \over {\mder /t P(w,t)}} \bigg|_{w=w_{14},t=a_2(f^{\pm})^{-1}}.
\end{equation}
In previous work with Pollack (\cite{BergdallPollack-FredholmSlopes}) we calculated $P(w,t)$ to a high accuracy using Koike's formula (see \cite{Robwebsite} also). From this, we  deduce the values given in Table \ref{table:derivatives-fourteen} below.
\begin{table}[htp]
\renewcommand{\arraystretch}{1.2}
\begin{center}
\caption{Valuations of derivatives of the coefficients of $P(w,t)$ at $w = w_{14}$.}\label{table:derivatives-fourteen}
\begin{tabular}{|r||c|c|c|c|c|c|c|c|c|}
\hline
$i$ & 1 & 2 & 3 & 4 & 5 & 6 & 7 & 8 & 9\\
\hline
$v_2(c_i'(w_{14}))$ & 0 & 7 & 19 & 34 & 60 & 78 & 106 & 140 & 179\\
\hline
$v_2(c_i'(w_{14})a_2(f^{\pm})^{-i})$ & -6 & -5 & 1 & 10 & 30 & 42 & 64 & 92 & 125\\
\hline
\end{tabular}
\end{center}
\end{table}

It thus seems likely that ${\mder /w} P(w,a_2(f^{\pm})^{-1})|_{w=w_{14}} = \sum c_i'(w_{14}) a_2(f^{\pm})^{-i}$  has $2$-adic valuation $-6$. To prove this, recall that a theorem of Buzzard and Kilford (\cite{BuzzardKilford-2adc}) implies that $c_i(w) \in (8,w)^{\lambda_i}\Z_2[\![w]\!]$ where $\lambda_i = {i+1\choose 2}$. Thus,
\begin{equation}\label{eqn:bk-inequality}
v_2(c_i'(w_{14})a_2(f^{\pm})^{-i}) \geq 3\left(\lambda_i - 1\right) - 6i
\end{equation}
for each $i$. The function on the right-hand side of \eqref{eqn:bk-inequality} is minimized at $i=3/2$ and at $i=3$ its value is  $-3$. So, the first few columns in Table \ref{table:derivatives-fourteen} do indeed imply  that
\begin{equation}\label{eqn:numer-valuation}
v_2\left({\mder /w} P(w,a_2(f^{\pm})^{-1})|_{w=w_{14}}\right) = v_2\left(c_1'(w_{14})a_2(f^{\pm})^{-1}\right) = -6.
\end{equation}
For the denominator in \eqref{eqn:log-derivative}, by Coleman's classicality theorem (\cite[Theorem 6.1]{Coleman-ClassicalandOverconvergent}) we have
\begin{equation*}
P(w_{14},t) = (1 - a_2(f^{\pm})t)(1-a_2(f^{\mp})t)\prod_{\beta}(1-\beta t)
\end{equation*}
where $v_2(\beta) \geq 13$ for each $\beta$. Then, the product rule implies that
\begin{equation}\label{eqn:denom-valuation}
{\mder /t}P(w_{14},t) \big|_{t=a_2(f^{\pm})^{-1}}
 = -2 a_2(f^{\pm})u
\end{equation}
where $u$ is a $2$-adic unit. From \eqref{eqn:numer-valuation} and \eqref{eqn:denom-valuation} we deduce that
\begin{equation*}
v_2\left(a_2'(w_{14})\over a_2(f^{\pm})\right) = -7,
\end{equation*}
as claimed.
\end{proof}

\section{Crystalline lifts with prescribed slope}\label{section:crys-lifts}
This section concerns the mod $p$ reduction of certain two-dimensional representations of $G = \Gal(\bar \Q_p/\Q_p)$. In Section \ref{sec:slope-bounds}, we will apply the result proved here to give a second bound on constant slope radii. Because of our intended application, we  make one global consideration (see Remark \ref{rmk:global-consideration}). Otherwise, our discussion is completely local. 

We write $\chi_{\cycl}$ for the cyclotomic character and $\bar \chi_{\cycl}$ for its mod $p$ reduction. We also write $I$ for the inertia subgroup of $G$. Throughout this section $V$ will generally mean a continuous, two-dimensional, $\bar \Q_p$-linear representation of $G$. To simplify notation, we use $\simeq_I$ to mean isomorphic as $I$-representations (and $\not\simeq_I$ accordingly).

In \cite{Fontaine-RepresentationSemiStable}, Fontaine defined what it means for $V$ to be crystalline (for general coefficients, see also \cite[Section 3]{BreuilMezard-Multiplicities}). A crystalline representation is uniquely determined by a certain two-dimensional $\bar \Q_p$-vector space $D_{\crys}(V)$ that is equipped with a filtration, called the Hodge filtration, and a linear operator $\varphi$, called the crystalline Frobenius. One consequence of the classification is that irreducible crystalline representations $V$ are parameterized up to twists by two numbers:\ first, an integer $k\geq 2$ and, second, an element $a_p \in \bar \Q_p$ such that $v_p(a_p) > 0$. (See \cite{BergerLiZhu-SmallSlopes} or \cite[Section 3]{Breuil-SomeRepresentations2} for discussion and references.) Given such a pair $(k,a_p)$, write $V_{k,a_p}$ for the corresponding representation. Concretely, we normalize this so that the Hodge filtration on $D_{\crys}(V_{k,a_p})$ has weights $0 < k-1$ and the Frobenius $\varphi$ is non-scalar with characteristic polynomial $X^2 - a_p X + p^{k-1}$. (We say the cyclotomic character has Hodge--Tate weight $-1$.)

For any $\bar \Q_p$-linear representation $V$ of $G$ there is a $\bar \Z_p$-linear and $G$-stable lattice $T \subset V$. We define $\bar V = (T \otimes_{\bar \Z_p} \bar \F_p)^{\ssfy}$ where the superscript ``$\ssfy$'' means to semi-simply the $G$-action on the $\bar \F_p$-vector space $T \otimes_{\bar \Z_p} \bar \F_p$.  It is well-known that $V \mapsto \bar V$ is independent of the choice of $T$. In what follows, we write $\bar V_{k,a_p}$ for what a pedant would write  $\bar{V_{k,a_p}}$.

Now fix a fundamental character $\omega_2$ of level 2 (\cite[Section 1.7]{Serre-PropertiesGaloisiennes}). This is a tame character on inertia and $\omega_2^{p+1} = \chi_{\cycl}$. Since $V$ is two-dimensional, \cite[Proposition 1]{Serre-SurLesRep} implies:
\begin{enumerate}
\item If $\bar V$ is irreducible, then $\bar V \simeq_I \omega_2^s \oplus \omega_2^{ps}$ for some integer $s$ not divisible by $p+1$.
\item If $\bar V$ is reducible then $\bar V \simeq_I \bar\chi_{\cycl}^{a} \oplus \bar \chi_{\cycl}^b$ for two integers $a,b$.
\end{enumerate}
\begin{notation}\label{defn:s(k,ap)}
If $k \geq 2$ and $v(a_p) > 0$, we write $s(k,a_p)$ for the choice of any integer according to the following two cases.
\begin{enumerate}
\item If $\bar V_{k,a_p}$ is irreducible and $\bar V_{k,a_p} \simeq_I \omega_2^s \oplus \omega_2^{ps}$, then $s(k,a_p) = s$.
\item If $\bar V_{k,a_p}$ is reducible, then $s(k,a_p) = 0$.
\end{enumerate}
\end{notation}
We call this notation, rather than a definition,  because there is ambiguity in case (a). Namely, since $\omega_2$ has order $p^2-1$, only the two values $s,ps \bmod{p^2-1}$ together are well-defined. That implies $\pm s(k,a_p) \bmod p+1$ is well-defined, giving meaning to the next proposition.

\begin{proposition}\label{prop:see-obstruction}
Let $k \geq 2$, $v_p(a_p) > 0$, and write $h = v_p(a_p)$. Then, for each integer $k'\geq 2$ such that
\begin{enumerate}
\item $k'-1 \not\congruent \pm s(k,a_p) \bmod p+1$ and
\item $\floor{{k'-2\over p-1}} < h$,
\end{enumerate}
we have $\bar V_{k',a_p'} \not\simeq_I \bar V_{k,a_p}$ for all $a_p'$ with $v_p(a_p') = h$.
\end{proposition}
In order to prove the proposition, we recall the following theorem of Berger, Li, and Zhu (\cite{BergerLiZhu-SmallSlopes}).
\begin{theorem}[Berger--Li--Zhu]\label{thm:blz}
Let $k \geq 2$, $v_p(a_p) > 0$, and assume that
\begin{equation}\label{eqn:blz}
\bfloor{{k-2\over p-1}} < v(a_p).
\end{equation}
Then, the following conclusions hold.
\begin{enumerate}
\item If $k-1 \not\congruent 0 \bmod p+1$, then $\bar V_{k,a_p} \simeq_I \omega_2^{k-1} \oplus \omega_2^{p(k-1)}$. In particular, $\bar V_{k,a_p}$ is irreducible.
\item If $k -1 \congruent 0 \bmod p+1$, then $\bar V_{k,a_p} \simeq_I (\chi_{\cycl}^{(k-1)/(p+1)})^{\oplus 2}$. In particular, $\bar V_{k,a_p}$ is reducible.
\end{enumerate}
\end{theorem}
\begin{proof}
Combine the main theorem and Proposition 4.1.4 from \cite{BergerLiZhu-SmallSlopes}.
\end{proof}

\begin{proof}[Proof of Proposition \ref{prop:see-obstruction}]
Throughout this proof we fix $k' \geq 2$ and suppose that $v_p(a_p') = h > \floor{{k'-2\over p-1}}$. We will argue depending on whether or not $\bar V_{k,a_p}$ is reducible. (We thank Sandra Rozensztajn for pointing out our argument when $\bar V_{k,a_p}$ is reducible extends to the case when $\bar V_{k,a_p}$ is irreducible.)

First suppose that $\bar V_{k,a_p}$ is reducible. By definition, $s(k,a_p) = 0$. So, the assumption (a) on $k'$ is that $k' - 1 \not\congruent 0 \bmod p+1$. But then the assumption (b)  and Theorem \ref{thm:blz}(a) together imply that $\bar V_{k',a_p'}$ is irreducible. In particular, $\bar V_{k,a_p} \not\simeq_I \bar V_{k',a_p'}$.

Now suppose the $\bar V_{k,a_p}$ is irreducible. If $\bar V_{k',a_p'}$ is reducible we are done. Otherwise $\bar V_{k,a_p}\simeq_I \omega_2^{s(k,a_p)} \oplus \omega_2^{ps(k,a_p)}$ whereas assumption (b) and Theorem \ref{thm:blz} implies that $\bar V_{k',a_p'} \simeq_I \omega_2^{k'-1} \oplus \omega_2^{p(k'-1)}$. Since $k'-1\not\congruent \pm s(k,a_p) \bmod p+1$ (by assumption (a)) we have that $\bar V_{k,a_p} \not\simeq_I \bar V_{k',a_p'}$. This completes the proof.
\end{proof}

For $k \in \Z$, $s \in \Z$, and $h > 0$ we define $X_{k,s,h}$ to be those integers $k'$ such that
\begin{enumerate}[(i)]
\item $k' \congruent k \bmod p-1$,
\item $k'-1 \not\congruent \pm s \bmod p+1$, and
\item $\floor{{k'-2\over p-1}} < h < {k'-2 \over 2}$.
\end{enumerate}
(Condition (iii) implies that $X_{k,s,h}$ is empty for $p=2$ and that $X_{k,s,h}$ has at most one element in it for $p=3$. See Section \ref{sec:optimization} for remarks on $p=2, 3$.) We consider the condition
\begin{equation}\label{eqn:star}
{\ceil{(p-3)h}-1 \over 3(p-1)} \geq 1
\end{equation}
and then define
\begin{equation}\label{eqn:mp(h)}
m_p(h) = \begin{cases}
\log_p\left({\ceil{(p-3)h} - 1 \over 3(p-1)}\right) & \text{if ${\ceil{(p-3)h}-1 \over 3(p-1)} \geq 1$};\\
0 & \text{otherwise.}
\end{cases}
\end{equation}
Here, $\log_p(-)$ means the logarithm with base $p$ ({\em not} the $p$-adic logarithm).
\begin{proposition}\label{prop:counting-prop}
Assume that $p > 3$ and fix any choice of $k$, $s$, and $h$ as above.
\begin{enumerate}
\item If $h$ satisfies \eqref{eqn:star}, then $X_{k,s,h}$ is non-empty.
\item Either $X_{k,s,h}$ is empty, or there exists a $k' \in X_{k,s,h}$ such that 
\begin{equation*}
v_p(k'-k) \geq \floor{m_p(h)}.
\end{equation*}
\end{enumerate}
\end{proposition}
During the proof of Proposition \ref{prop:counting-prop} we will need the following lemma.
\begin{lemma}\label{lemma:arith-prog}
Fix $k \in \mathbb{Z}$, an integer $n\geq 1$, and a non-empty open interval $I = (x_1,x_2)$ of length $\ell =x_2-x_1$. Then, the length of the arithmetic progression of integers that lie in $I$ and are congruent to $k$ modulo $n$ is at least $\floor{{\ceil{\ell}-1\over n}}$.
\end{lemma}
\begin{proof}
Clear.
\end{proof}
\begin{proof}[Proof of Proposition \ref{prop:counting-prop}]
The proof will show both statements simultaneously. To start, write $Y_{k,h}$ for the set of integers $k'$ satisfying just (i) and (iii) above. Thus $X_{k,s,h} \subset Y_{k,h}$. The set $Y_{k,h}$ is bounded by (iii), so use condition (i) to write
\begin{equation*}
Y_{k,h} = \set{k + m(p-1), k+(m+1)(p-1), \dotsc, k+(m+M-1)(p-1)}
\end{equation*}
for integers $m$ and $M$ with $M \geq 0$. By Lemma \ref{lemma:arith-prog} we can estimate $M$ based on $p$ and $h$. Let us omit that for now and deduce information about $X_{k,s,h}$.

Suppose that we know $M \geq 3p^t$ where $t\geq 0$ is an integer. We claim that this is enough to get a $k' \in X_{k,s,h}$ such that $v_p(k'-k) \geq t$ (in particular $X_{k,s,h}$ is non-empty). To see this, choose $0\leq j \leq p^t-1$ such that $v_p(m+j) \geq t$. Since $M \geq 3p^t$, the elements $k'_i = k + (m+j + ip^t)(p-1)$ for $i=0,1,2$ all lie in $Y_{k,h}$ and satisfy $v_p(k_i'-k)\geq t$. On the other hand, the $k_i'$ are also either consecutive odd or consecutive even integers modulo $p+1$. Since $p > 3$, they must be distinct and so at least one of them is in $X_{k,s,h}$, proving our claim.

Now we return to the assumption that $M \geq 3p^t$. By Lemma \ref{lemma:arith-prog} in fact we have $M \geq \floor{{\ceil{\ell} -1 \over p-1}}$ where $\ell$ is the length of the interval of $k'$ satisfying (iii). Certainly $\ell$ is at least $(p-3)h$, which is the length of the interval of $k'$ satisfying ${k'-2 \over p-1} < h < {k'-2\over 2}$. So, to find $M \geq 3p^t$ we may seek a $t$ such that 
\begin{equation}\label{eqn:final-seek}
\bfloor{{\ceil{(p-3)h}-1 \over p-1}} \overset{?}{\geq} 3p^t.
\end{equation}
Since $t$ is an integer, the floor on the left-hand side of \eqref{eqn:final-seek} can be removed. Dividing by $3$ and taking logarithms, we see that
\begin{equation*}
\floor{m_p(h)} \geq t \implies M \geq 3p^t.
\end{equation*}
So, the previous paragraph applied to $t = \floor{m_p(h)}$ completes the proof.
\end{proof}

\begin{theorem}\label{thm:galois-rep-thm}
Assume that $p > 3$, $k\geq 2$ and $h = v_p(a_p) > 0$ satisfies \eqref{eqn:star}. Then, there exists an integer $k'\geq 2$ such that
\begin{enumerate}
\item $k' \congruent k \bmod (p-1)p^{\floor{m_p(h)}}$
\item $\bar V_{k,a_p} \not\simeq_I \bar V_{k',a_p'}$ for all $v_p(a_p') = h$, and
\item $h < {k'-2\over 2}$.
\end{enumerate}
\end{theorem}
\begin{proof}
The set $X=X_{k,s(k,a_p),h}$ contains a $k'$ with $v_p(k' - k) \geq \floor{m_p(h)}$ by Proposition \ref{prop:counting-prop}. By conditions (i) and (iii) in the definition of $X$,  we have  $k' \congruent k \bmod (p-1)$ and $h < {k' -2\over 2}$. By conditions (ii) and (iii), and Proposition \ref{prop:see-obstruction}, if $k' \in X$ and $v_p(a_p') = h$,  then $\bar V_{k,a_p}\not\simeq_I \bar V_{k',a_p'}$.
\end{proof}

\begin{remark}\label{rmk:global-consideration}
The role of the hypothesis $h < {k'-2\over 2}$ in the previous theorem is to allow us to transfer between $T_p$-slopes and $U_p$-slopes (compare with Remark \ref{rmk:old-slope} and the proof of Theorem \ref{thm:slope-dependent}).
\end{remark}

\begin{remark}\label{remark:reducibility-into-account}
If we take into account the parity of $k$ and/or the reducibility of $\bar V_{k,a_p}$, then Theorem \ref{thm:galois-rep-thm} can be improved by constants. More specifically, if $\bar V_{k,a_p}$ is reducible then $s(k,a_p) = 0$, so the $\pm s \bmod p+1$ that needs to be avoided in $X_{k,s,h}$ is in fact only one number. Thus in the proof of Proposition \ref{prop:counting-prop} we can replace the instances of $3p^t$ by $2p^t$. 

If we assume that $\bar V_{k,a_p}$ is reducible and also that $k$ is even then we can  further replace $2p^t$ by $p^t$. Indeed, if $k$ is even then the condition $k' -1 \not \congruent 0 \bmod p+1$ is already implied by the condition that $k' \congruent k \bmod p-1$. 

To summarize, if $b=1,2,3$ and we consider the condition
\begin{equation}\label{eqn:star-b}
\tag{\ref{eqn:star}${}_b$}
{\ceil{(p-3)h}-1 \over b(p-1)} \geq 1
\end{equation}
and the function
\begin{equation}\label{eqn:mp(h)-b}
\tag{\ref{eqn:mp(h)}${}_{b}$}
m_p^{(b)}(h) = \begin{cases}
\log_p\left({\ceil{(p-3)h} - 1 \over b(p-1)}\right) & \text{if \eqref{eqn:star-b}};\\
0 & \text{otherwise},
\end{cases}
\end{equation}
then Theorem \ref{thm:galois-rep-thm} is proven with $b = 3$ in general, $b = 2$ when $\bar V_{k,a_p}$ is reducible and $b=1$ when $\bar V_{k,a_p}$ is reducible and $k$ is even.
\end{remark}

\section{Slope dependent bounds}\label{sec:slope-bounds}
The bound for $\CSw(f)$ in Theorem \ref{theorem:L-invariant-in-text} is highly dependent on the particular eigenform $f$. Here we discuss bounds that depend only on the slope of $f$ (what we call slope dependent bounds).

Before giving the positive result, we point out that slope dependent bounds are not a general phenomenon. More precisely, even in a fixed level $N$ there is no single function $m(h)$ such that $m(h) \goto \infty$ as $h \goto \infty$ and $m(h) \leq \CSw(f)$ when $f$ is a $p$-adic eigenform of tame level $N$ and slope $h$.

\begin{example}\label{example:no-slope-bounds}
Consider $p=2$ and $N=1$. Let $\mathcal W_+$ be the even component of $\mathcal W$ and $\mathcal W_+^{<3}$ be defined by $\kappa \in \mathcal W_+$ and $0 < v_p(w_\kappa) < 3$. Write $\mathcal C^{<3}$ for the preimage of the tame level 1, $2$-adic eigencurve above $\mathcal W_+^{<3}$. Buzzard and Kilford proved in \cite{BuzzardKilford-2adc} that $\mathcal C^{<3}$ decomposes into a union $\mathcal C^{<3} = \bigunion_{i} \mathcal C_i$ of irreducible components $\mathcal C_i$ such that ${\kappa}|_{\mathcal C_i}:\mathcal C_i \goto \mathcal W_+^{<3}$ is \'etale and if $x \in \mathcal C_i$, then $v_2(a_2(x)) = i v_2(w_{\kappa(x)})$.  In particular, if $x \in \mathcal C^{<3}$ corresponds to a finite slope $2$-adic eigenform $f_x$, then $\CSw(f_x) = v_p(w_{\kappa(x)})$. So, $\CSw(-)$ is constant on each weight fiber over $\mathcal W_{+}^{<3}$, whereas the list of slopes in each weight fiber is unbounded.
\end{example}

The literature contains more results, always related to the so-called ``spectral halo,'' that can be used to construct similar examples (\cite{Roe-Slopes,Kilford-5Slopes,KilfordMcMurday-7adicslopes,WanXiaoZhang-Slopes,LiuXiaoWan-IntegralEigencurves}).  Further, a conjecture of Pollack and the author (\cite{BergdallPollack-GhostPaperShort}) suggests that slope dependent bounds may never exist for $w_\kappa \nin \Z_p$. Below, we are going to show the converse statement is true. In particular, a slope dependent bound exists if we restrict to classical eigenforms.

In order to prove our result we need to discuss Galois representations, especially over eigencurves. Write $G_{\Q} = \Gal(\bar \Q/\Q)$ for the absolute Galois group of $\bar \Q$ and identify $G_{\Q_p}=\Gal(\bar \Q_p/\Q_p)$ with a decomposition group at $p$ inside $G_{\Q}$. As in the previous section, we write $I \subset G_{\Q_p}$ for the inertia subgroup.

Suppose that $g \in M_k(\Gamma_1(N))$ is an eigenform of weight $k\geq 2$ and let $\psi_g$ be its nebentype character. Then, there is a two-dimensional semi-simple and continuous representation $V_p(g)$ of $G_{\Q}$, with coefficients in $\bar \Q_p$, such that:
\begin{enumerate}
\item If $\ell \ndvd Np$, then $V_p(g)$ is unramified at $\ell$ and the characteristic polynomial of a geometric Frobenius element at $\ell$ is
\begin{equation*}
X^2 - a_\ell(g)X + \psi_g(\ell)\ell^{k-1}.
\end{equation*}
\item ${V_p(g)}|_{G_{\Q_p}}$ is crystalline. The vector space $D_{\crys}(V_p(g))$ has Hodge filtration with weights $0 < k-1$ and the crystalline Frobenius $\varphi$ on $D_{\crys}(V_p(g))$ is non-scalar with characteristic polynomial
\begin{equation*}
X^2 - a_p(g) X + \psi_g(p)p^{k-1}.
\end{equation*}
\end{enumerate}
The representations $V_p(g)$ are  completely determined by (a). For property (b), see \cite[Theorem 1.2.4(ii)]{Scholl-MotivesForModularForms} and \cite[Theorem 3.1]{ColemanEdixhoven-Semisimplicity}. If $v_p(a_p(g)) > 0$, then ${V_p(g)}|_{G_{\Q_p}}$ is irreducible and so one of the $V_{k,a_p}$ up to a twist. It is possible to make this more specific (see \eqref{eqn:modular-answer} in the next proof).
\begin{lemma}\label{lemma:inertial-restriction}
Suppose that $g \in M_k(\Gamma_1(N))$ is an eigenform of weight $k\geq 2$ and $v_p(a_p(g)) > 0$. Then there exists an $a_p \in \bar \Q_p$ such that $v_p(a_p) = v_p(a_p(g))$ and $(\bar V_p(g)|_{G_{\Q_p}})^{\ssfy} \simeq_I \bar V_{k,a_p}$.
\end{lemma}
\begin{proof}
View the Dirichlet character $\psi_g$ instead as a character on $G_{\Q}$ using class field theory. Then by \cite[Th\'eor\`eme 6.5]{Breuil-SomeRepresentations2} where exists an unramified (because $N$ is prime to $p$) character $\psi_g^{1/2}$ of $G_{\Q_p}$ whose square is $\psi_g|_{G_{\Q_p}}$ and such that
\begin{equation}\label{eqn:modular-answer}
\restrict{V_p(g)}{G_{\Q_p}} \simeq V_{k,a_p(g)\psi_g(p)^{1/2}} \otimes \psi_g^{1/2}.
\end{equation}
In particular, this implies the result with $a_p = a_p(g)\psi_g(p)^{1/2}$ (note that $\psi_g$ takes values in the roots of unity).
\end{proof}

The construction of Galois representations extends $p$-adic analytically to eigencurves. Specifically, for each $x \in \mathcal C_N$ there exists a two-dimensional semi-simple and continuous representation $V_x$ of $G_{\Q}$ that is unramified away from $Np$ and such that a geometric Frobenius element at $\ell \ndvd Np$ acts with characteristic polynomial
\begin{equation*}
X^2 - a_\ell(x) X + \langle \ell \rangle(x) \ell^{-1}\kappa(x)(\ell).
\end{equation*}
This uniquely determines $V_x$ as before. For instance, if $f \in M_k(\Gamma)$ is a stabilized eigenform $f=g_\alpha$, then we have $V_{x_f} = V_p(g)$.

For a moment, consider any continuous representation $V$ of $G_{\Q}$ over $\bar \Q_p$. We use, as in Section \ref{section:crys-lifts}, the notation $\bar V$ to denotes its semi-simplification modulo $p$. It is a fact (see \cite{AshStevens-Duke} for example) that for a fixed $N$ there are only finitely many representations of the form $\bar V_x$ as $x$ runs over points of $\mathcal C_N$. In fact, $\mathcal C_N$ decomposes into a finite union
\begin{equation*}
\mathcal C_N = \bigunion_{\bar V} \mathcal C_N(\bar V)
\end{equation*}
of open and closed rigid analytic subspaces $\mathcal C_N(\bar V)$ characterized as those points $x$ such that $\bar V_x \simeq \bar V$. (The $\mathcal C_N(\bar V)$ are not necessarily irreducible though.)

\begin{lemma}\label{lemma:constancy}
If $U$ is a $p$-adic family (passing through some point on $\mathcal C_N$), then $u\mapsto \bar V_u$ is constant on $U$.
\end{lemma}
\begin{proof}
By \cite[Corollary 2.2.9]{Conrad-IrredComponents} the family $U$ is contained in a unique irreducible, and thus unique connected, component of $\mathcal C_N$. Thus $U$ is completely contained in one of the $\mathcal C_N(\bar V)$ appearing in the above discussion.
\end{proof}

Now we can prove the main result of this section. Recall that we defined $m_p(h)$ in \eqref{eqn:mp(h)} and it includes the condition \eqref{eqn:star} on $h$. Also recall  the ``$k$-normalized'' constant slope radius $\CSk(-)$ from Definition \ref{defn:CSk}.

\begin{theorem}\label{thm:slope-dependent}
Assume that $p > 3$, $f\in M_{\kappa}^{\dagger}(\Gamma)$ satisfies \eqref{hyp:et}, and $w_{\kappa} \in \Z_p$. Set $h := v_p(a_p(f))$. Then, $\floor{m_p(h)} \leq \CSk(f)$.
\end{theorem}
\begin{proof}
If $h$ does not satisfy \eqref{eqn:star} then $m_p(h) = 0$ so the result is trivial. Thus in the proof we will assume that \eqref{eqn:star} is satisfied.

Consider a constant slope $p$-adic family $U$ passing through $f$. Since $w_\kappa \in \Z_p$, there are arbitrarily large integers $k$ such that $k$ and $\kappa$ have non-empty weight fibers in $U$ and $h < {k-2\over 2}$. By Coleman's classicality theorem (\cite[Theorem 6.1]{Coleman-ClassicalandOverconvergent}) the weight $k$ point in $U$ is classical. In fact, this point is necessarily a stabilization $g_\alpha$ of an eigenform $g$ of level $\Gamma_1(N)$ and such that  $0 < h = v_p(\alpha) = v_p(a_p(g)) < {k-2\over 2}$ (by Remark \ref{rmk:old-slope}). Taking $k$ approaching to $\kappa$, we may assume that $f=g_\alpha$. 

By Lemma \ref{lemma:inertial-restriction} we may choose $a_p$ such that $v_p(a_p) = h$ and
\begin{equation}\label{eqn:what-i-want}
(\bar V_{x_f}|_{G_{\Q_p}})^{\ssfy} \simeq_I \bar V_{k,a_p}
\end{equation}
Since $h$ satisfies \eqref{eqn:star}, Theorem \ref{thm:galois-rep-thm} now guarantees the existence of an integer $k' \geq 2$ such that 
\begin{enumerate}
\item $k'\congruent k \bmod (p-1)p^{\floor{m_p(h)}}$,
\item $h < {k'-2\over 2}$, and
\item $\bar V_{k,a_p} \not\simeq_I \bar V_{k',a_p'}$ for any choice of $v(a_p') = h$.
\end{enumerate}
We claim that $U$ cannot pass over the weight $k'$. Suppose it did. Then, (b) would imply that above the weight $k'$ we have an eigenform $f'$ of the form $f' = g'_{\alpha'}$ with $g' \in M_k(\Gamma_1(N))$ (using the argument via Coleman's theorem again). Furthermore, $0 < h = v_p(a_p(g')) = v_p(\alpha')$, so by Lemma \ref{lemma:inertial-restriction} we know there exists an $a_p'$ with $v(a_p') = h$ such that \eqref{eqn:what-i-want} holds with $f$ replaced by $f'$, $k$ replaced by $k'$, and $a_p$ replaced by $a_p'$. Now property (c) in our choice of $k'$ implies that
\begin{equation}
(\bar V_{x_f}|_{G_{\Q_p}})^{\ssfy}\not\simeq_I (\bar V_{x_{f'}}|_{G_{\Q_p}})^{\ssfy},
\end{equation}
which contradicts Lemma \ref{lemma:constancy}. So, $U$ does not pass over the weight $k'$, completing the proof.
\end{proof}

\begin{remark}
If $p=2$ then to get classical weights $k$ approaching $\kappa$ we would need to assume $w_\kappa \in 4\Z_2-8\Z_2$ if $\kappa$ is an odd character and $w_\kappa \in 8\Z_2$ if $\kappa$ is an even character. (We ignored this because we needed $p>3$ anyways.)
\end{remark}

\begin{remark}\label{rmk:exclusionary-proof}
The proof of Theorem \ref{thm:slope-dependent} is exclusionary in the following way. If $f \in M_k(\Gamma)$ is an eigenform of slope $h$ then what we showed is that (for $h$ large enough) there is some weight $k' \congruent k \bmod (p-1)p^{\floor{m_p(h)}}$ such that any constant slope $p$-adic family passing through $f$ must omit the weight $k'$. In particular, if we are concerned only with integer $p$-adic weights, then the furthest a constant slope $p$-adic family through $f$ can extend is to weights $k' \congruent k \bmod (p-1)p^{\floor{m_p(h)}+1}$.
\end{remark}

\begin{remark}
This article has used ``classical'' to mean level $\Gamma_1(N) \cap \Gamma_0(p)$, though classical eigenforms of level $\Gamma_1(Np)$ are also overconvergent $p$-adic eigenforms, just of a slightly modified weight. More specifically, if $\chi$ is the $p$-part of the nebentype character (seen as a finite order character on $\Z_p^\times$) of an eigenform of weight $k$ and level $\Gamma_1(Np)$ then the corresponding $p$-adic weight is $\kappa(z) = z^k\chi(z)$. Since $\chi$ is trivial on the infinite pro-cyclic part of $\Z_p^\times$, we see $w_\kappa = w_k \in \Z$. Thus, Theorem \ref{thm:slope-dependent} applies to such eigenforms (of non-critical slope) as well. On the other hand, Theorem \ref{thm:slope-dependent} does not apply the finite slope classical eigenforms of level $\Gamma_1(Np^r)$, with $r \geq 2$, because the $p$-adic weights of such forms are not in the closure of $\Z$ (cf.\ Example \ref{example:no-slope-bounds}).
\end{remark}

\section{Optimizing slope dependent bounds}\label{sec:optimization}

The slope dependent bound we proved in Theorem \ref{thm:slope-dependent} is not optimal. It is based on the analysis in Section \ref{section:crys-lifts} for which there may be improved statements. For instance, Remark \ref{remark:reducibility-into-account} already provides some conditions to improve the bound by a constant factor. 

It is also not out of the question that, in Theorem   \ref{thm:blz}, the condition \eqref{eqn:blz} can be replaced by
\begin{equation}\label{eqn:modified-blz}
\tag{\ref{eqn:blz}*} 
\bfloor{{k-1\over p+1}} < v_p(a_p).
\end{equation}
See, for instance, the remarks following \cite[Conjecture 2.1.1]{BuzzardGee-Slopes} and the data in \cite[Section 6]{Rozensztajn-2dlocus}. Recent joint work with Levin (\cite{BergdallLevin-BLZ}) also aims to improve Theorem \ref{thm:blz}, providing a result between \eqref{eqn:modified-blz} and \eqref{eqn:blz}.

Assuming \eqref{eqn:modified-blz}, the arguments in Section \ref{section:crys-lifts} go through without change for $p>3$, and Theorem \ref{thm:galois-rep-thm} can be altered by replacing $m_p(h)$ by
\begin{equation}\label{eqn:mp(h)star}
\tag{\ref{eqn:mp(h)}*}
m_p^{\ast}(h) = \begin{cases}
\log_p\left({\ceil{(p-1)h} - 2 \over 3(p-1)}\right) & \text{if ${\ceil{(p-1)h} - 2 \over 3(p-1)} \geq 1$;}\\
0 & \text{otherwise.}
\end{cases}
\end{equation}
The same trichotomy as in Remark \ref{remark:reducibility-into-account} also applies.

One benefit of improving Theorem \ref{thm:blz} would be to soundly include the cases $p=2$ and $p=3$ where the numerical data is easier to come by. To get an analog of Theorem \ref{thm:galois-rep-thm} in those cases, some adjustments to Proposition \ref{prop:counting-prop} need to be made (since $p>3$ is used in the middle of that proof). We leave that to the reader. We do observe, though, that the proof in Proposition \ref{prop:counting-prop} does go through for $p=2,3$ as long as $k$ is even and $\bar V_{k,a_p}$ is reducible (by Remark \ref{remark:reducibility-into-account}).

As an example of this discussion, assume that we can replace \eqref{eqn:blz} in Theorem \ref{thm:blz} with \eqref{eqn:modified-blz}. Then, assume that $f$ is an regular eigenform of non-critical slope $h$, even weight, and such that ${\bar V_{x_f}}|_{G_{\Q_p}}$ is reducible. Some basic estimates then give
\begin{equation}\label{eqn:optimal-implication}
h \in \Z_{\geq 3} \implies \begin{cases}
\floor{\log_p(h-1)} \leq \CSk(f) & \text{ if $p\geq 3$;}\\
\floor{\log_p(h-2)} \leq \CSk(f) & \text{ if $p=2$,}
\end{cases}
\end{equation}
improving Theorem \ref{thm:slope-dependent}.

\begin{remark}
An examination of the proof of Theorem \ref{thm:slope-dependent} shows that to deduce \eqref{eqn:optimal-implication} we could also assume the precise global conjecture that would follow from an affirmative answer to \cite[Question 4.9]{Buzzard-SlopeQuestions} (see the text surrounding \cite[Conjecture 2.1.1]{BuzzardGee-Slopes} as well).
\end{remark}
 
 \section{Comparison to Gouv\^ea--Mazur-like quantities}\label{section:gouvea-mazur-comparison}
 
 The purpose of this final section is to discuss the relationship between our results and a former conjecture of Gouv\^ea and Mazur (\cite{GouveaMazur-FamiliesEigenforms}). The Gouv\^ea--Mazur conjecture was shown to be false by Buzzard and Calegari (\cite{BuzzardCalegari-GouveaMazur}), but it remains open what kind of salvage is possible. We indicated earlier that the lower bound in Theorem \ref{theorem:L-invariant-in-text} is heuristically close to the behavior suggested in \cite{GouveaMazur-FamiliesEigenforms}, so it remains to determine how closely.
 
 Let us recall the setup of the conjecture in \cite{GouveaMazur-FamiliesEigenforms} now. The quantity we consider is $d(k,h)$, which is the multiplicity of $h$ as the $p$-adic valuation of an eigenvalue for the $U_p$-operator acting on $M_k(\Gamma)$. Coleman showed (\cite{Coleman-pAdicBanachSpaces}) that there is some function $\GM(h)$ such that if $k \congruent k' \bmod (p-1)p^{\GM(h)}$ then $d(k,h) = d(k',h)$. It was proven by Wan (\cite{Wan-GouveaMazur}) that one could assume that $\GM(h) = O(h^2)$ with implicit constants depending on $N$ and $p$. The original conjecture of Gouv\^ea--Mazur was that $\GM(h) = \ceil{h}$ is sufficient. This is what is shown to be false in \cite{BuzzardCalegari-GouveaMazur}.
 
We note some finiteness properties of the function $\GM(h)$ are not available for the function $\CS^k(-)$. Specifically, even if we fix the slope $h$ then the maximum value of $\CS^k(f)$, ranging over all classical eigenforms of slope $h$, may be $\infty$ in the presence of ramification of the weight map in slope $h$. However, we will ignore that and proceed to present evidence that one {\em might} want to confuse $\GM(h)$ and constant slope radii. Or, at least, we will compare our lower bounds with plausible values for $\GM(h)$ while restricting to the special, but important, case of $p$-new forms. 

For most of the rest of this section we let $p=5$ and $N=1$. If $k\geq 2$ is an even integer and $f \in S_k(\Gamma_0(5))$ is a newform, then Theorems \ref{theorem:L-invariant-normalization} and \ref{thm:slope-dependent} provide obstructions to the existence of a constant slope $5$-adic family passing through $f$ in terms of $\mathscr L_f$ and the slope $h_k = {k-2\over 2}$. Taking into account Remark \ref{rmk:exclusionary-proof} and Section \ref{sec:optimization} we have that, for $k\geq 8$, a constant slope family through $f$ should only  exist over $k' \congruent k \bmod {4\cdot 5^{m_5(f,h_k)}}$ where
\begin{equation}\label{eqn:cong-obstruction}
m_5(f,h_k) = \max(\floor{\log_5(h_k-1)} + 1, \ceil{-v_5(\mathscr L_f)}).
\end{equation}
(Note:\ All of the mod $5$ Galois representations at level $\Gamma_0(5)$ are reducible, even globally.) Now we need some data on the $m_5(f,h_k)$'s, given in Table \ref{table:5adic-gouvea-mazur-example}, that was provided to us by Robert Pollack (as discussed in Example \ref{example:single-L-invariant-example}).

\begin{small}
\begin{table}[htp]
\renewcommand{\arraystretch}{1.1}
\caption{Lists of $5$-adic congruence obstructions $m_5(f,h_k)$ to constant slope $5$-adic families through newforms of given weight $k$.}
\begin{center}
\begin{tabular}{|c|l|l|}
\hline
$k$ & List of $v_5(\mathscr L_f)$'s & List of $m_5(f,h_k)$'s\\
\hline
8 &  $0, -2, -2$ & $1,2,2$\\
10 & $2, -2, -2$ & $1,2,2$\\
12 & $-1, -2, -2$ & $1,2,2$\\
14 & $-1, -2, -2, -4, -4$ & $2,2,2,4,4$\\
16 & $-1, -3, -3, -4, -4$ & $2,3,3,4,4$\\
18 & $-1, -2, -2, -4, -4$ & $2,2,2,4,4$\\
20 & $-1, -2, -2, -5, -5, -6, -6$ & $2,2,2,5,5,6,6$\\
22 & $-1, -2, -2, -5, -5, -6, -6$ & $2,2,2,5,5,6,6$\\
24 & $-1, -2, -2, -4, -4, -7, -7$ & $2,2,2,4,4,7,7$\\
26 & $-1, -3, -3, -4, -4, -7, -7, -8, -8$ & $2,3,3,4,4,7,7,8,8$\\
\hline
\end{tabular}
\end{center}
\label{table:5adic-gouvea-mazur-example}
\end{table}%
\end{small}

As an example of what to do with this, suppose that $k=16$. The list of $v_5(\mathscr L_f)$'s in that case is $\set{-1,-3,-3,-4,-4}$ and so the list of $m_5(f,7)$'s is $\set{2,3,3,4,4}$. This means, at the very least, that no constant slope $7$ family should exist over all weights $k' \congruent 16 \bmod 4\cdot 5$. Further, if the $m_5(f,h_k)$'s are exactly the obstruction to families existing, we expect one family over weights $k' \congruent 16 \bmod {4 \cdot 5^2}$, two more over weights $k' \congruent 16 \bmod 4\cdot 5^3$, and finally two more still over $k' \congruent 16 \bmod 4\cdot 5^4$. To see if this is plausible, we give the slopes of the $U_5$-operator acting in weights near to 16 in Table \ref{table:U5-slopes}. 
\begin{small}
\begin{table}[htp]
\renewcommand{\arraystretch}{1.1}
\caption{Slopes of $U_5$-operator at weights $k$ near to $16$.}
\begin{center}
\begin{tabular}{|c|l|c|}
\hline
$k$ & Slopes of $U_5$ acting on $S_k(\Gamma_0(5))$ & $d(k,7)$\\
\hline
20 & $1, 9, 9, \dotsc$ & $0$\\ 
36 & $1, 4, 5, 17, 17, \dotsc$ & $0$\\
116 & $1, 5, 6, 7, 8, 9, 14, \dotsc$ & $1$\\
516$^\ast$ & $1, 6, 7, 7, 7, 8, 14, 15, \dotsc$ & $3$\\
2516$^\ast$ & $ 1, 7, 7, 7, 7, 7, 14, 15, \dotsc$ & $5$\\
12516$^\ast$ & $ 1, 7, 7, 7, 7, 7, 14, 15, \dotsc$ & $5$\\
\hline
\end{tabular}
\end{center}
\label{table:U5-slopes}
\end{table}%
\end{small}

As far as we know, this is as close as a human can come to ``seeing'' constant slope families. So, the data indicates a strong link between the $m_5(f,7)$'s and the corresponding $5$-adic families passing through the $f$'s. The $\ast$'s appearing in Table \ref{table:U5-slopes} indicate that we used work of Lauder (\cite{Lauder-Computations}) to calculate in the corresponding weights, as opposed to using in-built {\tt sage} commands to calculate actual spaces of cuspforms, and that we did not make the calculation provably correct. (This would have involved making some constants effective, but possibly increasing the length of the computation.) Using Lauder's work is crucial to make further calculation since, as best as we can tell, the largest value of $-v_p(\mathscr L_f)$ in weight $k$ is linear in $k$, meaning that we must be prepared to consider weights on the order of $k + (p-1)p^{k}$. In fact we compiled data as in Table \ref{table:U5-slopes} for each $8 \leq k \leq 26$ and summarized it in Table \ref{table:gouvea-mazur-quantities}. For notation, if $j\geq 0$ we define $d_j(k) = d(k+(p-1)p^j, h_k)$.

\begin{small}
\begin{table}[htp]
\renewcommand{\arraystretch}{1.1}
\setlength{\tabcolsep}{4pt}
\caption{Multiplicity $d_j(k)$ of the slope $h_k={k-2\over 2}$ at weight $k + 4\cdot 5^j$.}
\begin{center}
\begin{tabular}{|c|c|c|c|c|c|c|c|c|c|c|}
\hline
$k$ & $d_0(k)$ & $d_1(k)$ & $d_2(k)$ & $d_3(k)^{\ast}$ & $d_4(k)^{\ast}$ & $d_5(k)^{\ast}$& $d_6(k)^{\ast}$& $d_7(k)^{\ast}$& $d_8(k)^{\ast}$& $d_9(k)^{\ast}$\\
\hline
$8$ & 0 & 1 & 3 & 3 & 3 & 3 & 3 & 3 & 3 & 3\\
$10$ & 0 & 1 & 3 & 3 & 3 & 3 & 3 & 3 & 3 & 3\\
$12$ & 0 & 0 & 3 & 3 & 3 & 3 & 3 & 3 & 3 & 3\\
$14$ & 0 & 0 & 3 & 3 & 5 & 5 & 5 & 5 & 5 & 5\\
$16$ & 0 & 0 & 1 & 3 & 5 & 5 & 5 & 5 & 5 & 5\\
$18$ & 0 & 0 & 3 & 3 & 5 & 5 & 5 & 5 & 5 & 5\\
$20$ & 0 & 0 & 3 & 3 & 3 & 5 & 7 & 7 & 7 & 7\\
$22$ & 0 & 0 & 3 & 3 & 3 & 5 & 7 & 7 & 7 & 7\\
$24$ & 0 & 0 & 3 & 3 & 5 & 5 & 5 & 7 & 7 & 7\\
$26$ & 0 & 0 & 1 & 3 & 5 & 5 & 5 & 7 & 9 & 9\\
\hline
\end{tabular}
\end{center}
\label{table:gouvea-mazur-quantities}
\end{table}%
\end{small}

The link we are after is clear now:\ Tables \ref{table:5adic-gouvea-mazur-example} and \ref{table:gouvea-mazur-quantities} show that, except for the case $k=12$, the list of $m_5(f,h_k)$'s is exactly the list of $j$ where $d_j(k) > d_{j-1}(k)$ {\em with} multiplicity counted by $d_j(k) - d_{j-1}(k)$. This numerical correlation was replicated in all the cases where we had access to the $\mathscr L$-invariants. One could also replace the function $d_j(k)$ by a more robust function like the minimum of $d(k+u(p-1)p^j, h_k)$ where $u=1,2,\dotsc,p-1$. We did that as well, but there seemed to be no difference in the data.

We have not found a theoretical reason why the case $k=12$ is slightly off. The same thing (being off ``by one") happened for $k=52$ and not again in our test range ($k\leq 142$). We note, however, that the disagreement in the data is occurring only for the single supposedly {\em largest} of the $5$-adic families through newforms of weights $12$, $52$, etc.\ and the size of the family is controlled by the logarithmic term in \eqref{eqn:cong-obstruction}.

We now release the assumption that $p=5$ and $N=1$. The numerics detailed above deserve expanding upon. Based on the data we have available, we can offer a question that now seems likely to have an affirmative answer.

\begin{question}\label{question:final-q}
Does there exist a non-negative function $\twid m_p(h)$ that grows like $\log h$ as $h \goto \infty$ and satisfies the following property? 

\begin{quote}
If $f \in S_k(\Gamma)$ is a $p$-new eigenform then define
\begin{equation*}
\twid m_p(f) = \max(\floor{\twid m_p(h)}, \ceil{-v_p(\mathscr L_f)}).
\end{equation*}
Then, the list of $\twid m_p(f)$'s as $f$ ranges over $p$-new eigenforms in $S_k(\Gamma)$ is equal to the list of $j$ such that $d_j(k) > d_{j-1}(k)$ with multiplicity counted by $d_j(k) - d_{j-1}(k)$.
\end{quote}
\end{question}
One can also ask the same question without the assumption that $f$ is $p$-new, replacing the $\mathscr L$-invariant by the logarithmic derivative of $a_p$. We do not yet have a strong feeling on that because data in this case is much harder to come by.

It seems unlikely that a negative answer to Question \ref{question:final-q} can be provided as stated, so let us end with a falsifiable conjecture. The data we have strongly points an answer to Question \ref{question:final-q}, in tame level 1 for small primes, with a specific function $\widetilde m_p(h)$.

\begin{conjecture}\label{conj:final-conj}
For $3 \leq p\leq 11$ and $N=1$, Question \ref{question:final-q} has an affirmative answer witnessed by $\twid m_p(h) = \floor{\log_p(h)} + 1$ (when $h \geq 1$).
\end{conjecture}

Admittedly, the scope of this conjecture is small. But, we have not found a theoretical explanation for why the ``$h-1$'' in \eqref{eqn:cong-obstruction} needs to be replaced by ``$h$'' in order to make Conjecture \ref{conj:final-conj} true. If we had one, or if we had a lot more data, we would surely attempt to present a more broad conjecture. Recent developments (\cite{Graef-Linvariant,ABGT-Linvariants}) should make it possible for interested researchers to gather more data, including in some $p$-old cases.

\bibliography{const_slope_bib}
\bibliographystyle{abbrv}

\end{document}